\documentclass[12pt, reqno]{amsart}
\usepackage[utf8]{inputenc}
\usepackage[margin=2.5cm]{geometry}
\usepackage{graphicx}
\usepackage{amsmath}
\usepackage{amsthm}
\usepackage{amssymb,bbm}%
\usepackage[numbers, square]{natbib}
\usepackage{makecell}
\usepackage[hidelinks]{hyperref}

\usepackage{tikz}
\usepackage{subcaption}

\theoremstyle{plain}
\newtheorem{theorem}{Theorem}

\newtheorem{corollary}[theorem]{Corollary}
\newtheorem{proposition}[theorem]{Proposition}

\theoremstyle{definition}

\theoremstyle{remark}

\makeindex
\makeglossary

\newcommand{\E}{\mathbb E\,}
\newcommand{\R}{\mathbb{R}}

\newcommand{\C}{\mathbb{C}}

\newcommand{\s}{\mathbb{S}}
\newcommand{\B}{\mathbb{B}^d}

\newcommand{\diag}{\mathop{\mathrm{diag}}\nolimits}

\newcommand{\conv}{{\mathrm{conv}}}

\newcommand{\bx}{\mathbf{x}}
\newcommand{\by}{\mathbf{y}}
\newcommand{\bz}{\mathbf{z}}
\newcommand{\bu}{\mathbf{u}}
\newcommand{\bv}{\mathbf{v}}
\newcommand{\ba}{\mathbf{a}}
\newcommand{\bb}{\mathbf{b}}
\newcommand{\be}{\mathbf{e}}

\newcommand{\dd}{{\rm d}}

\usepackage{xcolor}

\begin{document}

\author[A.~Gusakova]{Anna Gusakova}
\address{Anna Gusakova, Institute of Mathematical Stochastics, M\"unster University, Germany}
\email{gusakova@uni-muenster.de}

\author[E.~Spodarev]{Evgeny Spodarev}
\address{Evgeny Spodarev, Institute of Stochastics, Ulm University, Germany}
\email{evgeny.spodarev@uni-ulm.de}

\author[D.~Zaporozhets]{Dmitry Zaporozhets}
\address{Dmitry Zaporozhets, St.~Petersburg Department of Steklov Institute of Mathematics, Russia}
\email{zap1979@gmail.com}

\title[Intrinsic volumes of ellipsoids]{Intrinsic volumes of ellipsoids}
\thanks{The work of ES was supported by EIMI (Leonard Euler International Mathematical Institute, Saint Petersburg), grant 075-15-2019-1620 as of 08/11/2019. AG  was supported by the DFG under Germany's Excellence Strategy  EXC 2044 -- 390685587, \textit{Mathematics M\"unster: Dynamics - Geometry - Structure}. The work of DZ was supported by the Foundation for the Advancement of Theoretical Physics and Mathematics “BASIS” and by RFBR and DFG according to the research project № 20- 51-12004.
}


\begin{abstract}
We deduce explicit formulae for the intrinsic volumes of an ellipsoid in $\mathbb R^d$, $d\ge 2$, in terms of elliptic integrals. Namely, for an ellipsoid  ${\mathcal E}\subset \mathbb R^d$ with semiaxes $a_1,\ldots, a_d$ we show that
 
\begin{align*}
V_k({\mathcal E})=\kappa_k\sum_{i=1}^da_i^2s_{k-1}(a_1^2,\dots,a_{i-1}^2,a_{i+1}^2,\dots,a_d^2)\int_0^{\infty}{t^{k-1}\over(a_i^2t^2+1)\prod_{j=1}^d\sqrt{a_j^2t^2+1}}\,\rm{d}t
\end{align*}
for all $k=1,\ldots,d$,
where $s_{k-1}$ is the $(k-1)$-th elementary symmetric polynomial and $\kappa_k$ is the volume of the $k$-dimensional unit ball.
Some examples of the intrinsic volumes $V_k$ with low and high $k$ are given where our formulae look particularly simple. {As an application we derive new formulae for the expected $k$-dimensional volume of random $k$-simplex in an ellipsoid and random Gaussian $k$-simplex.}
\end{abstract}

\maketitle
\section{Introduction and main result}
For a non-empty convex compact set $K\subset \R^d$, consider its parallel set  of radius $r>0$ defined as
 $K+ r\B$, where $\B$ is the $d$-dimensional unit ball and the operation of {\it Minkowski addition} means the pointwise sum of two sets. 
The well-known {\it Steiner formula} writes the volume $|\cdot|_d$ (i.e., the Lebesgue measure) of $K+r\B$ as a polynomial of degree $d$ in $r$:
\begin{align}\label{1800}
    |K+ r\B|_d=\sum\limits_{k=0}^d \kappa_{d-k}V_{k}(K) r^{d-k}, \quad r\ge 0,
\end{align}
where $\kappa_k:=\pi^{k/2}/\Gamma\left(\frac k2+1\right)$ is the  volume of the $k$-dimensional unit  ball. The coefficients $V_k(K)$, $k=0,\ldots, d,$ above are called {\it intrinsic volumes} of $K$.  They are normalized in a way that  if $K$ is $k$--dimensional, then $V_k(K)$ coincides with the $k$-dimensional volume of $K$.

The intrinsic volumes  as well as related  Minkowski functionals
(or querma{\ss}integrals) and tensors of (poly)convex 
bodies  and sets with positive reach   play an important role in convex geometry and in the applied fields, such as fractal and topological data analysis, compare e.g. \cite{MeckeStoyan00,MeckeStoyan02,Torq02,Planck2016,collaboration2015planck,SpoStrWin15}.

Although intrinsic volumes are basic and fundamental quantities of convex bodies, their computation is not a simple task even for those classical shapes like ellipsoids.
So far, only indirect results containing the computation of the surface area as R hypergeometric function \cite{Carlson66}, an Abelian integral \cite{Tee05} or Lauricella hypergeometric function \cite{Rivin07} and the expression of  intrinsic volumes in terms of Gaussian determinants  \cite{KZ12} are available. Thus, Theorem 1.1 of \cite{KZ12} states that for an arbitrary ellipsoid ${\mathcal E}\subset\R^d$ with semiaxes $a_1,\ldots, a_d$ we have
$$
V_k({\mathcal E})=\frac{(2\pi)^{k/2}}{k!} \E \sqrt{ \det(\langle A \xi_i, A\xi_j \rangle)_{i,j=1}^k}, 
$$
where $A=\mathrm{diag}(a_1,\ldots, a_d)$ and $\xi_1,\dots, \xi_k$ are i.i.d. standard Gaussian vectors in $\R^d$. Decomposing $\xi_i$ into independent spherical and radial parts as
$
	\xi_i=\eta_i\cdot\|\xi_i\|
$
and noting that $\eta_1,\dots, \eta_d$ are i.i.d. random vectors in $\R^d$ uniformly distributed on the unit sphere $\s^{d-1}$ equipped with  the Hausdorff measure $\sigma(\cdot)$ normalized by $\sigma(\s^{d-1})$ yields
\begin{align}\label{1113b}
	V_k({\mathcal E})&=\frac{(2\pi)^{k/2}}{k!} \E\|\xi_1\|\dots \E\|\xi_k\|\cdot\E \sqrt{ \det(\langle A\eta_i, A\eta_j \rangle)_{i,j=1}^k}
	\\\notag&
	=\frac{1}{k!\kappa_{d-1}^k} \int_{(\s^{d-1})^k} \sqrt{\det(\langle A \bu_i, A\bu_j \rangle)_{i,j=1}^k}\, \sigma(\dd\bu_1)\ldots \sigma(\dd\bu_k),
\end{align}
where in the second equation we used that
\begin{align*}
    \E \|\xi_i\|=\frac{\sqrt{2}\,\Gamma(\frac{d+1}{2})}{\Gamma(\frac{d}{2})},\quad \sigma(\s^{d-1})=\frac{2\pi^{d/2}}{\Gamma(\frac{d}{2})}.
\end{align*}
The later $k$-fold integration in~\eqref{1113b} makes an explicit computation of $V_k({\mathcal E})$  particularly complex.

The problem of deriving explicit formulae for $V_k({\mathcal E})$ is also deeply connected to the hypothesis that ellipsoids are uniquely determined (up to a rigid motion) by their intrinsic volumes. It is solved positively so far only in $d=2,3$ \cite{PetTar20} as well as for the dual volumes \cite{myroshnychenko2020unique}. 

The main result of our paper gives the formula for $V_k(\mathcal E)$ in terms of one-dimensional elliptic integrals. Before formulating it,  for a tuple $(t_1,\dots,t_n)$ denote by $s_m(t_1,\dots,t_n)$ (where $m\leq n$) the $m$-th elementary symmetric polynomial of $t_1,\dots,t_n$  defined as
\begin{align*}
    s_{m}(t_1,\dots,t_n):=\sum_{1\leq i_1<\ldots<i_m\leq n}\prod_{j=1}^mt_{i_j}.
\end{align*}
Now we are ready to formulate the main theorem.
\begin{theorem}\label{0948}
For every $k\in\{1,\ldots,d\}$ we have
\begin{align*}
V_k({\mathcal E})=\kappa_k\sum_{i=1}^da_i^2s_{k-1}(a_1^2,\dots,a_{i-1}^2,a_{i+1}^2,\dots,a_d^2)\int_0^{\infty}{t^{k-1}\over(a_i^2t^2+1)\prod_{j=1}^d\sqrt{a_j^2t^2+1}}\dd t.
\end{align*}
\end{theorem}

Using a substitute of variables $t\to \sqrt t$, the integrals in the above formula can be expressed in terms of the hypergeometric R-function (cf. \cite[6.8-8, p. 184]{Carlson77}) 
$$
R_{-s}(\bb,\bz)=\frac{1}{B\left(s,\sum_{j=1}^d b_j - s\right)}  
\int_0^{\infty}{t^{s-1}  \dd t \over \prod_{j=1}^d(1+z_j t )^{b_j}},
$$
where $B(\cdot,\cdot)$ is the usual Beta function and
\begin{align*}
    \bb=(b_1,\ldots,b_d)\in \R^d,\quad
    \bz=(z_1,\ldots,z_d)\in \left( \C \setminus (-\infty, 0]\right)^d, \quad
    0<s<\sum_{j=1}^d b_j.
\end{align*}
\begin{corollary}
For every $k\in\{1,\ldots,d\}$ we have
$$
V_k({\mathcal E})= \frac{\kappa_k}{2}  B\left({ d+2-k \over 2}, {k\over 2}  \right) \sum_{i=1}^da_i^2s_{k-1}(a_1^2,\dots,a_{i-1}^2,a_{i+1}^2,\dots,a_d^2)    R_{-k/2} \left(\be_i+\frac12\sum_{j=1}^d\be_j,\ba\right),
$$
where $\ba=(a_1^2,\ldots,a_d^2)$ and $\be_1,\dots,\be_d$ is the standard orthonormal basis in $\R^d$.

\end{corollary}


To obtain Theorem \ref{0948}, we derived an auxiliary formula expressing $V_k(\mathcal E)$ in terms of the integrals over the unit sphere which might be of independent interest.
\begin{theorem}\label{1406}
For every $k\in\{1,\ldots,d\}$ we have
\begin{equation}\label{eq:IntrVolE_main}
V_k({\mathcal E})={1\over k\kappa_{d-k}}\sum_{i=1}^da_i^2s_{k-1}(a_1^2,\dots,a_{i-1}^2,a_{i+1}^2,\dots,a_d^2)\int\limits_{\s^{d-1}}{u_i^2\over h^{k}_{\mathcal E}(\bu)}\,\sigma(\dd\bu),
\end{equation}
where $\bu=(u_1,\dots,u_d)$ and $h_{\mathcal E}(\bu)=\sqrt{a^2_1u^2_1+\dots+a^2_du^2_d}$ is the support function of $\mathcal E$ (see Section~\ref{sect:Basic}).
\end{theorem}
{
Notice that the formulae from Theorem \ref{0948} and Theorem \ref{1406} are valid for $k=d$ as well. In this case we have $V_d({\mathcal E})=\kappa_da_1\cdot\ldots\cdot a_d$, which means that some of our formulae can presumably be further simplified} {for some $k$, see Section~\ref{sect:Examples}.}

Theorem~\ref{0948} readily follows from Theorem~\ref{1406} and the following proposition applied with $\alpha=2,\beta=k$. The idea of its proof is taken from~\cite[Lemma~2]{myroshnychenko2020unique}.
\begin{proposition}\label{0831}
For $i\in\{1,\dots, d\}$ and $\alpha, \beta\in\R^1$ such that $\beta>0$, $\alpha>d-\beta$ we have
\begin{align}\label{1112}
    \int\limits_{\s^{d-1}}{|u_i|^\alpha \over h^{\beta}_{\mathcal E}(\bu)}\,\sigma(\dd\bu)
    =
     {4\pi^{(d-1)/2} \Gamma(\frac{\alpha+1}{2})\over \Gamma({d+\alpha-\beta\over 2})\Gamma({\beta\over 2})}\int_0^{\infty}{t^{\beta-1}\over(a_i^2t^2+1)^{\alpha/2}\prod_{j=1}^d\sqrt{a_j^2t^2+1}}\dd t.
\end{align}
\end{proposition}

The paper is organized as follows: 
 In Section \ref{sect:Examples},  examples of the  intrinsic volumes $V_k({\mathcal E})$ of low ($k=1,2$) and high ($k=d-1,d-2$) orders  are given.  {In Section \ref{sect:ExpVol}, our  Theorem \ref{0948} is applied to get a more explicit form of the  expected $k$--dimensional volume of convex hulls of $k+1$ {idependent} random points uniformly distributed in an ellipsoid or having {an arbitrary centered} Gaussian distribution law.  } Section~\ref{sect:Basic}
contains some  preliminaries from convex and differential geometry which are used in the proofs of our results located in
Section  \ref{1122}.


\section{Examples}\label{sect:Examples}

It is well known that $V_0({\mathcal E})=1$, $V_d({\mathcal E})=\kappa_d a_1\cdot\ldots\cdot a_d.$ From~\eqref{eq:IntrVolE_main} for $k=1$ we get
\begin{align}\label{1058}
    V_1(\mathcal E)={1\over \kappa_{d-1}}\sum_{i=1}^da_i^2\int\limits_{\s^{d-1}}{u_i^2\over h_{\mathcal E}(\bu)}\,\sigma(\dd\bu)={1\over \kappa_{d-1}}\int\limits_{\s^{d-1}}\sqrt{\sum_{i=1}^d {a_i^2}{u_i^2}}\,\sigma(\dd\bu),
\end{align}
which agrees with~\eqref{1113b} and with the Kubota formula for the intrinsic volumes: in case $k=1$
it states that 
\begin{align*}
    V_1(K)={1\over \kappa_{d-1}}\int\limits_{\s^{d-1}}{ h_{K}(\bu)}\,\sigma(\dd\bu)
\end{align*}
for any convex body $K$.
On the other hand, from Theorem~\ref{0948} we obtain
\begin{align*}
    V_1({\mathcal E})=2\sum_{i=1}^d\int_0^{\infty}{a_i^2\over(a_i^2t^2+1)\prod_{j=1}^d\sqrt{a_j^2t^2+1}}\dd t.
\end{align*}

Taking $k=2$ in~\eqref{eq:IntrVolE_main} gives
\begin{align*}
    V_2({\mathcal E})&={1\over 2\kappa_{d-2}}\sum_{i=1}^d\bigg[a_i^2\bigg(\sum_{j=1}^da_j^2-a_i^2\bigg)\int\limits_{\s^{d-1}}{u_i^2\over h^{2}_{\mathcal E}(\bu)}\,\sigma(\dd\bu)\bigg]
    \\
    &=\pi\sum\limits_{i=1}^d a_i^2-\frac{\pi}{\sigma(\s^{d-1}) }  
\int\limits_{\s^{d-1}}\frac{\sum\limits_{i=1}^d a_i^4u_i^2}{\sum\limits_{i=1}^d a_i^2u_i^2}  \sigma(\dd\bu),
\end{align*}
where we used that $\sigma(\s^{d-1})=2\pi\kappa_{d-2}$. Applying Proposition~\ref{0831} with $\alpha=\beta=2$ leads to
\begin{align*}
     V_2({\mathcal E})=\pi\sum\limits_{i=1}^d a_i^2-\pi\sum_{i=1}^d\int_0^{\infty}{a_i^4t\over(a_i^2t^2+1)\prod_{j=1}^d\sqrt{a_j^2t^2+1}}\dd t.
\end{align*}

Now let us use the following  duality relation for ellipsoids which can be found in \cite[Prop. 4.8]{KZ16}:
\begin{equation}\label{eq:Dual}
V_k({\mathcal E})=\frac{\kappa_k}{\kappa_d \kappa_{d-k}  }
V_d({\mathcal E})V_{d-k}({\mathcal E}^o),
\quad k=0,\ldots,d,
\end{equation}
where ${\mathcal E}^o$ is the ellipsoid dual to ${\mathcal E}$:
\begin{align*}
    {\mathcal E}^o=\{ x\in\R^d:\langle x,y\rangle \le 1, \quad y\in {\mathcal E}\}.
\end{align*}
Using this relation and the fact that $\mathcal E^o$ has semiaxes $a_1^{-1},\dots,a_d^{-1}$ we can easily derive the formulae {for $V_{d-1}({\mathcal E})$ and $V_{d-2}({\mathcal E})$ from the formulae for $V_{1}({\mathcal E}^o)$ and $V_{2}({\mathcal E}^o)$, respectively}:
\begin{align*}
    V_{d-1}(\mathcal E)&={a_1\dots a_d\over 2}\int\limits_{\s^{d-1}}\sqrt{\sum_{i=1}^d {a_i^{-2}}{u_i^2}}\,\sigma(\dd\bu)
    \\
    &=\kappa_{d-1}a_1^2\dots a_d^2\sum_{i=1}^d\int_0^{\infty}{\dd t \over(t^2+a_i^2)\prod_{j=1}^d\sqrt{t^2+a_j^2}}
\end{align*}
and
\begin{align*}
    V_{d-2}({\mathcal E})&=\kappa_{d-2}a_1\dots a_d\sum\limits_{i=1}^d a_i^{-2}-\frac{a_1\dots a_d}{2\pi}  
    \int\limits_{\s^{d-1}}\frac{\sum\limits_{i=1}^d a_i^{-4}u_i^{2}}{\sum\limits_{i=1}^d a_i^{-2}u_i^2}  \sigma(\dd\bu)
    \\
    &=\kappa_{d-2}a_1\dots a_d\sum\limits_{i=1}^d a_i^{-2}-
    \kappa_{d-2}a_1^2\dots a_d^2
    \sum_{i=1}^d\int_0^{\infty}{a_i^{-2}t\over(t^2+a_i^2)\prod_{j=1}^d
    \sqrt{t^2+a_j^2}}\dd t.
\end{align*}

\section{Expected volumes of random simplices}\label{sect:ExpVol}

The intrinsic volumes of ellipsoids have a remarkable connection to the average volume of random $k$-simplex, which is formed as a convex hull of $k+1$ isotropic random points. More precisely, let $X_0,\ldots, X_k$, $1\leq k\leq d$, be random points in $\R^d$, whose joint distribution is invariant with respect to rotations.  
We consider their convex hull $\conv(X_0,\ldots, X_k)$, which is a random (possibly degenerate) simplex, and its $k$-dimensional volume $|\conv(X_0,\ldots, X_k)|_k$, which is a well-defined random variable. In \cite[Corollary 1.1]{GGZ}, it was shown that for any non-degenerate matrix $A\in\R^{d\times d}$ we have
\begin{equation}\label{eq_simpl_ell}
\E |\conv(AX_0,\ldots, AX_k)|_k={\kappa_{d-k}\over {d\choose k}\kappa_d}V_k(\mathcal{E}_A)\cdot\E|\conv(X_0,\ldots, X_k)|_k,
\end{equation}
where $\mathcal{E}_A:=\{\bx\in\R^d\colon \bx^{\top}(A^\top A)^{-1}\bx\le 1\}$ is an ellipsoid.

There are two particularly interesting models which formula \eqref{eq_simpl_ell} can be applied to. For the first model consider a bit more general setup. Let $K\subset \R^d$ be a $d$-dimensional convex body, and let $Y_0,\ldots, Y_k$, $1\leq k\leq d$, be independent random points uniformly distributed over $K$. The classical problem to evaluate $M_k(K):=\E |\conv(Y_0,\ldots, Y_k)|_k$ for given $K$  goes back to Klee \cite{Klee69}. By now, not so many exact formulae for $M_k(K)$ have been  obtained, and those mostly for $d=2$ and $d=3$. In dimension $2$, the exact formulae for $M_2(K)$ are available for triangles \cite{Reed74}, regular planar $n$-gones \cite{Buch84} and   parallelograms \cite{Reed74}. In dimension $3$, the formulae for $M_3(K)$ have been derived in case of tetrahedron \cite{Mann94} and cube \cite{Zin03}. In arbitrary dimension $d$, $M_d(K)$ is known only for the ball \cite{Kin96} and, hence, due to affine invariance for any ellipsoid. The situation is more complicated if $k<d$ since $M_k(K)$ is not affine invariant anymore. For any $d$ and any $1\leq k\leq d$, the exact formula for $M_k(K)$ is known only in case of a ball \cite{Miles71} (see also \cite[Theorem 8.2.3]{SW08}). For the functional $M_1(K)$ describing the average distance between two points chosen uniformly in $K$, some formulae are known in planar case \cite{bG51, ZP, Baesel21, tS85}. A formula for the cube \cite{BBC07} (via the so-called {\it box integral}) is also available in dimension $d=3$. However, the box integral does not have a closed form expression  for $d\ge 4$. Heinrich \cite{Hein14} has also obtained a representation of $M_1(\mathcal{E})$ for a $d$-dimensional ellipsoid $\mathcal{E}$ with semi-axes $0<a_1\leq \ldots \leq a_d$ in terms of the elliptic integral \eqref{1058}. Applying Theorem \ref{0948} to relation \eqref{eq_simpl_ell}, we are able to obtain more {explicit} formulae for $M_k(\mathcal{E})$ and all $1\leq k\leq d$:

\begin{theorem}
Let $Y_0,\ldots, Y_d$ be independent random points uniformly distributed in $\mathcal{E}$. Then for any $k\in \{1,\ldots, d\}$, the expected volume $M_k(\mathcal{E}):=\E|\conv(Y_0,\ldots, Y_k)|_k$ equals
\begin{align*}
M_k(\mathcal{E})&={1\over{ 2^k \Gamma({k\over 2}+1)}}{\Gamma\big({(d+1)(k+1)\over 2}+1\big)\over \Gamma \big({(d+1)(k+1)\over 2}+{1\over 2}\big)}\Big({\Gamma({d\over 2}+1)\over \Gamma({d+1\over 2}+1)}\Big)^{k+1}\sum_{i=1}^da_i^2s_{k-1}(a_1^2,\dots,a_{i-1}^2,a_{i+1}^2,\dots,a_d^2)\\
&\qquad\qquad\times\int_0^{\infty}{t^{k-1}\over(a_i^2t^2+1)\prod_{j=1}^d\sqrt{a_j^2t^2+1}}\dd t.
\end{align*}
\end{theorem}

\begin{proof}
The result follows directly from \eqref{eq_simpl_ell} applied to $X_0,\ldots, X_k$ distributed uniformly inside the unit ball $\mathbb{B}^d$ and an affine transformation mapping $\mathbb{B}^d$ to $\mathcal{E}$. Substituting the exact values for $\E|\conv(X_0,\ldots, X_k)|_k$ from \cite[Theorem 8.2.3]{SW08} (see also \cite[Corollary 1.5]{GGZ}) and the formula for $V_k(\mathcal{E})$ from Theorem \ref{0948} finishes the proof. It should be noted that in order to simplify the constant we have used the Legendre duplication formula $\Gamma(z)\Gamma(z+1/2)=2^{1-2z}\sqrt{\pi}\Gamma(2z)$.
\end{proof}

The second model is the so-called Gaussian random simplex. Let $X_0,\ldots, X_k$, $1\leq k\leq d,$ be independent standard Gaussian random vectors in $\R^d$. Their convex hull $\conv(X_0,\ldots, X_k)$ is almost surely a $k$-simplex. Its expected $k$-dimensional volume was calculated by Miles \cite[Equation (70)]{Miles71}. Using \eqref{eq_simpl_ell} the later result can be generalized to the convex hull of non-standard Gaussian random vectors.

\begin{theorem}
Let $Y_0,\ldots, Y_d$ be independent Gaussian centered random vectors in $\R^d$ with non--degenerate covariance matrix $\Sigma$. Let $\lambda_1,\ldots, \lambda_d>0$ be the eigenvalues of $\Sigma$. For any $k\in \{1,\ldots, d\}$ we have
\begin{align*}
\E|\conv(Y_0,\ldots, Y_k)|_k&={\sqrt{k+1}\over\Gamma({k\over 2}+1)2^{k/2}}\sum_{i=1}^d\lambda_is_{k-1}(\lambda_1,\dots,\lambda_{i-1},\lambda_{i+1},\dots,\lambda_d)\\
&\qquad\qquad\times\int_0^{\infty}{t^{k-1}\over(\lambda_it^2+1)\prod_{j=1}^d\sqrt{\lambda_jt^2+1}}\dd t.
\end{align*}
\end{theorem}

\begin{proof}
Let $X_0,\ldots, X_k$ be independent standard Gaussian random vectors. Consider an affine transformation with matrix $\Sigma^{1/2}$ (which  is well defined since $\Sigma$ is symmetric and positive definite). It is clear that {$\Sigma^{1/2}X_i$}, $0\leq i\leq k$, are independent centered Gaussian random vectors with covariance matrix $\Sigma$. Thus,
{
first applying~\eqref{eq_simpl_ell} and then~\cite[Equation (70)]{Miles71} we have
\begin{align*}
\E&|\conv(Y_0,\ldots, Y_k)|_k=\E|\conv(\Sigma^{1/2}X_0,\ldots, \Sigma^{1/2}X_k)|_k
\\
&={\kappa_{d-k}\over {d\choose k}\kappa_d}V_k(\mathcal{E}_\Sigma)\cdot\E|\conv(X_0,\ldots, X_k)|_k
={\kappa_{d-k}\over {d\choose k}\kappa_d}V_k(\mathcal{E}_\Sigma)\cdot{2^{k/2}\sqrt{k+1}\over k!}{\Gamma({d+1\over 2})\over \Gamma({d-k+1\over 2})},
\end{align*}
}where $\mathcal{E}_\Sigma:=\{\bx\in\R^d\colon \bx^{\top}\Sigma^{-1}\bx\leq 1\}$ {is an ellipsoid}
with semi-axes  $a_i=\sqrt{\lambda_i}>0$, $1\leq i\leq d$. Thus, by Theorem \ref{0948} we conclude that
\begin{align*}
\E|\conv(Y_0,\ldots, Y_k)&|_k={(d-k)!\over d!}{2^{k/2}\sqrt{k+1}\over \Gamma({k\over 2}+1)}{\Gamma({d\over 2}+1)\Gamma({d+1\over 2})\over \Gamma({d-k\over 2}+1)\Gamma({d-k+1\over 2})}\\
&\times\sum_{i=1}^d\lambda_is_{k-1}(\lambda_1,\dots,\lambda_{i-1},\lambda_{i+1},\dots,\lambda_d)\int_0^{\infty}{t^{k-1}\over(\lambda_it^2+1)\prod_{j=1}^d\sqrt{\lambda_jt^2+1}}\dd t.
\end{align*}
Applying the Legendre duplication formula $\Gamma(z)\Gamma(z+1/2)=2^{1-2z}\sqrt{\pi}\Gamma(2z)$ twice finishes the proof.
\end{proof}

\section{Basic facts from convex geometry}\label{sect:Basic}

 The {\it mixed volumes} $V(K_1,\ldots, K_d)$ of convex bodies $K_1,\ldots, K_d \subset \R^d$
 are introduced via the \emph{Minkowski theorem} as the coefficients of the polynomial expansion
\begin{align}\label{1801}
 |r_1 K_1+\ldots+ r_d K_d|_d=\sum\limits_{k_1,\ldots,k_d=1}^d r_{k_1} \ldots r_{k_d} V(K_{k_1},\ldots, K_{k_d}), \quad r_1,\ldots, r_d\ge 0,
\end{align}
which is the generalisation of the Steiner formula, see~\eqref{1800}. They are non-negative and symmetric with respect to all permutations of the indices of $K_1,\ldots, K_d$, cf. e.g. \cite[Theorem 5.1.6]{rS14}. From~\eqref{1800} and~\eqref{1801} we immediately have 
\begin{align}\label{1805}
     V_k(K)=\frac 1{\kappa_{d-k}}{d\choose k} W_{d-k}(K), \quad k=0,\ldots,d,
\end{align}
where
\begin{align}\label{2026}
    W_k(K)= V\Big(\underbrace{K, \ldots, K}_{k}, \underbrace{\B,\ldots,\B}_{d-k} \Big)
\end{align}
are called the \emph{querma{ss}integrals} of $K$, see more on them e.g. in \cite[Chapter 4]{rS14}. 

Let $h_{K}(\bx)=\sup_{\by\in K} \langle \bx,\by \rangle$, $\bx\in\R^d$ be the support function of a convex body $K$.  
For an ellipsoid 
\begin{align}\label{1138}
    {\mathcal E}=\bigg\{ \bx=(x_1,\ldots, x_d)\in\R^d: \sum_{i=1}^d \frac{x_i^2}{a_i^2}\le 1  \bigg\}\quad \text{we have}\quad
h_{{\mathcal E}}(\bx)=\sqrt{\sum_{i=1}^d a_i^2 x_i^2}.
\end{align}

The main ingredient of the proof of Theorem~\ref{1406} is the formula expressing the mixed volume of convex bodies in terms of their support functions. To formulate it, for arbitrary matrices $Q_1,\dots,Q_{d-1}\in\R^{(d-1)\times (d-1)}$ denote by $D(Q_1,\dots,Q_{d-1})$ their \emph{mixed discriminant} defined as (see, e.g.,~\cite{Bapat89})
\begin{align}\label{0950}
    D_{d-1}(Q_1,\dots, Q_{d-1})=\frac{1}{(d-1)!}\sum_{\tau\in S_{d-1}}\det Q(\tau),
\end{align}
where $S_{d-1}$ is the symmetric group on $d-1$ elements, $\tau=(\tau(1),\dots,\tau(d-1))$, and $Q(\tau)$ is a matrix whose $i$-th column coincides with the $i$-th column of $Q_{\tau(i)}$.
Essentially, we compose a matrix using the columns of $Q_1,\dots, Q_{d-1}$ with pairwise different indices according to a random permutation (uniformly chosen from $S_{d-1}$), and then take its expected determinant. Given an arbitrary matrix $Q\in\R^{d\times d}$ denote by $Q^{i}$ its principal minor produced by the deleting the $i$-th row and the $i$-th column from $Q$. 

Now suppose that $K_1,\dots,K_d$ are some convex bodies having $C^2$--smooth boundaries with positive Gaussian curvatures at each point. For a convex body $K\subset\R^d$ denote by  $H_{K}(\bx)\in\R^{d\times d}$ the Hessian matrix of its support function $h_K$ at point $\bx$. 
Then it is known (see e.g. \cite[p. 1061]{leicht93}) that
\begin{align*}
 V\left(K_1, \ldots, K_d\right)=\frac{1}{d}
 \sum_{i=1}^d\int\limits_{\s^{d-1}} h_{K_1}(\bu) D_{d-1}\left(H_{K_2}^{i}(\bu),\dots,H_{K_{d}}^{i}(\bu)\right)
 \,\sigma(\dd\bu).
\end{align*}
Combining this with~\eqref{1805} and~\eqref{2026} for $k\geq 1$ leads to
\begin{align}\label{2038}
 V_k(K)&=\frac{{d\choose k}}{d\kappa_{d-k}}
 \sum_{i=1}^d\int\limits_{\s^{d-1}}h_{K}(\bu) D_{d-1}\Big(\underbrace{H_{K}^{i}(\bu),\dots,H_{K}^{i}(\bu)}_{k-1}, \underbrace{H_{{\B}}^{i}(\bu),\dots,H_{{\B}}^{i}(\bu)}_{d-k} \Big)\,\sigma(\dd\bu).
\end{align}

\section{Proofs}\label{1122}
\subsection{Proof of Theorem~\ref{1406}}
We are going to apply~\eqref{2038} with $K=\mathcal E$. To this end, let us first calculate the mixed discriminant under the integral in the right-hand side of~\eqref{2038}. 
It follows by straightforward double differentiation of the right-hand side of~\eqref{1138} that $$H_{\mathcal E}(\bu)= h_{{\mathcal E}}^{-3}(\bu) A(\bu),$$ where
\begin{align*}
A(\bu)&= 
h_{\mathcal E}^2(\bu) \diag(a_1^2,\dots,a_d^2)-(a_1^2u_1,\ldots,a_{d}^2u_{d})^{\top}(a_1^2u_1,\ldots,a_{d}^2u_{d})
\\[10pt]
&=\left(\begin{array}{cccc}
   a_1^2 (h_{\mathcal E}^2(\bu) - a_1^2 u_1^2) & -a_1^2 a_2^2 u_1 u_2 & \ldots & -a_1^2 a_d^2 u_1 u_d    \\
     - a_2^2 a_1^2u_2 u_1 &  a_2^2 (h_{\mathcal E}^2(\bu) - a_2^2 u_2^2) & \ldots & -a_2^2 a_d^2 u_2 u_d    \\
     \ldots &  \ldots & \ldots & \ldots    \\
       - a_d^2 a_1^2 u_d u_1 & - a_d^2a_{2}^2  u_d u_{2}& \ldots &  a_d^2 (h_{\mathcal E}^2(\bu) - a_d^2 u_d^2 ) \end{array}
\right).
\end{align*}
In particular, recalling that $\bu\in\s^{d-1}$ we have
\begin{align*}
    H_{{\B}}(\bu)&=\|\bu\|^{-3}(I_d-\bu\bu^\top)=I_d-\bu\bu^\top
    \\[10pt]
    &=\left(\begin{array}{cccc}
   1- u_1^2 & - u_1 u_2 & \ldots & - u_1 u_{d}    \\
     - u_1 u_2 &  1-u_2^2 & \ldots & - u_2 u_{d}    \\
     \ldots &  \ldots & \ldots & \ldots    \\
       - u_1 u_{d}  & - u_{2} u_{d} & \ldots &  1- u_{d}^2
\end{array}
\right).
\end{align*}
Using this and the linearity of the mixed discriminants we arrive at
\begin{align}\label{0456}
D_{d-1}\Big(&\underbrace{H_{\mathcal E}^{i}(\bu),\dots,H_{\mathcal E}^{i}(\bu)}_{k-1}, \underbrace{H_{{\B}}^{i}(\bu),\dots,H_{{\B}}^{i}(\bu)}_{d-k} \Big)
\\\notag
 &=h_{\mathcal E}^{3-3k}(\bu) D_{d-1}\Big(\underbrace{A^i(\bu),\dots,A^i(\bu)}_{k-1}, \underbrace{H^i_{{\B}}(\bu),\dots,H^i_{{\B}}(\bu)}_{d-k} \Big).
\end{align}
By~\eqref{0950} we have
\begin{align}\label{1605}
    D_{d-1}\Big(\underbrace{A^i(\bu),\dots,A^i(\bu)}_{k-1}, \underbrace{H^i_{{\B}}(\bu),\dots,H^i_{{\B}}(\bu)}_{d-k} \Big)=\frac{1}{(d-1)!}\sum_{\tau\in S_{d-1}} q^{(i)}_\tau(\bu),
\end{align}
where $q^{(i)}_\tau(\bu)$ is the determinant of the matrix composed of $k-1$ columns of $A^i(\bu)$ and $d-k$ columns of $H^i_{{\B}}(\bu)$ chosen according to the permutation $\tau$. Since for  $i=1,\dots,d$ and $\tau\in S_{d-1}$ all $q^{(i)}_\tau(\bu)$ look similar, it is enough to compute only one of them:  the rest will be derived by  changing  the variables. Consider for simplicity the one corresponding to $i=d$ and to the identical permutation $q^{(d)}_{\mathbf{id}}(\bu)$. We have
\begin{align*}
    q^{(d)}_{\mathbf{id}}(\bu)=\det\left(\begin{array}{ccc|ccc}
   a_1^2 (h_{\mathcal E}^2(\bu) - a_1^2 u_1^2) & \ldots & -a_1^2 a_{k-1}^2 u_1 u_{k-1}  & -u_1u_k&\dots &-u_1u_{d-1}\\ 
   \ldots&\ldots&\ldots&\ldots&\ldots&\ldots\\
     - a_{k-1}^2a_1^2  u_{k-1}u_1 & \ldots& a_{k-1}^2 (h_{\mathcal E}^2(\bu) - a_{k-1}^2 u_{k-1}^2) &-u_{k-1}u_k& \ldots & - u_{k-1} u_{d-1}    \\[6pt]
     \hline \rule{0pt}{1\normalbaselineskip}
     - a_k^2a_1^2  u_ku_1 &\ldots&-a_k^2 a_{k-1}^2  u_ku_{k-1}&1-u_k^2&\ldots&
     - u_ku_{d-1}
     \\
     \ldots &  \ldots & \ldots & \ldots & \ldots & \ldots    \\
       -a_{d-1}^2a_1^2   u_{d-1} u_1 & \ldots&- a_{d-1}^2 a_{k-1}^2 u_{d-1}u_{k-1} &-u_{d-1}u_k& \ldots &  1-u_{d-1}^2
\end{array}
\right).
\end{align*}
Taking  the factor $a_j^2$ out of the $j$-th column, $j=1,\dots,k-1,$ leads to
\begin{equation} \label{1945}
q^{(d)}_{\mathbf{id}}(\bu)=\prod_{j=1}^{k-1}a^2_j\cdot \det\left( \begin{array}{c|c}
     C_1 & C_2 \\\hline
     C_3 & C_4
\end{array}\right),
\end{equation}
where
\begin{equation}\label{0623}
\begin{aligned}
C_1&=h_{\mathcal E}^2(\bu) I_{k-1}-(a_1^2u_1,\ldots,a_{k-1}^2u_{k-1})^{\top}(u_1,\ldots,u_{k-1})
= h_{\mathcal E}^2(\bu)I_{k-1}-\bv_1^\top \bu_1,
\\
C_2&=-(u_1,\ldots,u_{k-1})^{\top}(u_k,\ldots,u_{d-1})
=-\bu_1^\top \bu_2,
\\
C_3&=-(a_k^2u_k,\ldots,a_{d-1}^2u_{d-1})^{\top}(u_1,\ldots,u_{k-1})
=-\bv_2^\top \bu_1,
\\
C_4&=I_{d-k}-(u_k,\ldots,u_{d-1})^{\top}(u_k,\ldots,u_{d-1})
=I_{d-k}-\bu_2^\top \bu_2.
\end{aligned}
\end{equation}
Here, we used the notation
\begin{align*}
    \bu_1=(u_1,\ldots,u_{k-1}),\quad &\bu_2=(u_k,\ldots,u_{d-1}),\\
    \bv_1=(a_1^2u_1,\ldots,a_{k-1}^2u_{k-1}),\quad &\bv_2=(a_k^2u_k,\ldots,a_{d-1}^2u_{d-1}).
\end{align*}
Matrix $C_4$ is invertible with 
\begin{equation}\label{eq:invC4}
C_4^{-1}=I_{d-k}+{\bu_2^\top \bu_2\over 1-\|\bu_2\|^2}.
\end{equation}
Using the well-known formula for the determinant of a block matrix (see
e.g.~\cite[Lemma~5]{LM17}), we obtain
\begin{equation}\label{eq:detBlock}
q^{(d)}_{\mathbf{id}}(\bu)=\prod_{j=1}^{k-1}a^2_j\cdot\det(C_1-C_2 C_4^{-1} C_3) \det C_4.  
\end{equation}

First of all, using the Weinstein-Aronszajn identity
\begin{equation}\label{eq:WAidentity}
\det(I_m+M_1M_2)=\det(I_n+M_2M_1),\quad M_1\in\R^{m\times n}, M_2\in\R^{n\times m},
\end{equation}
with $M_1=(1-\|\bu_2\|^2)^{-1}\bu_2^\top , M_2=\bu_2$
we obtain from~\eqref{eq:invC4} that
\begin{equation}\label{eq:detC4}
\det C_4=1-\|\bu_2\|^2.
\end{equation}
Further using~\eqref{0623}  and \eqref{eq:invC4}, we calculate
\begin{align*}
    C_2 C_4^{-1} C_3&=\bu_1^\top \bu_2\left(I_{d-k}+{\bu_2^\top \bu_2\over 1-\|\bu_2\|^2}\right)\bv_2^\top \bu_1
    \\&=\left(1+{\|\bu_2\|^2\over 1-\|\bu_2\|^2}\right)\bu_1^\top(\bu_2\bv_2^\top) \bu_1
    =\lambda\,\bu_1^\top \bu_1,
\end{align*}
where 
\begin{align*}
    \lambda:={\bu_2\bv_2^\top\over 1-\|\bu_2\|^2}\in\R.
\end{align*}
Therefore, 
\begin{align*}
    \det(C_1-C_2 C_4^{-1} C_3)&= \det\left(
    h_{\mathcal E}^2(\bu)I_{k-1}-\bv_1^\top \bu_1-\lambda\,\bu_1^\top \bu_1\right)
    \\
    &=h_{\mathcal E}^{2k-2}(\bu)\det\left(I_{k-1}-h_{\mathcal E}^{-2}(\bu)(\bv_1^\top+\lambda\bu_1^\top)\bu_1\right).
\end{align*}
Again applying the Weinstein-Aronszajn identity with $M_1=\bv_1^\top+\lambda\bu_1^\top$ and $M_2=\bu_1$ and noting that $h_{\mathcal E}^{2}(\bu)=\bu_1\bv_1^\top+\bu_2\bv_2^\top+a_d^2u_d^2$ leads to
\begin{align*}
    \det(C_1-C_2 C_4^{-1} C_3) &=h_{\mathcal E}^{2dk-2}(\bu)\left(1-h_{\mathcal E}^{-2}(\bu)\bu_1(\bv_1^\top+\lambda\bu_1^\top)\right)
    \\&=h_{\mathcal E}^{2k-4}(\bu)\left(h_{\mathcal E}^{2}(\bu)-\bu_1\bv_1^\top-\lambda\bu_1\bu_1^\top\right)
    \\&=h_{\mathcal E}^{2k-4}(\bu)\left(\bu_2\bv_2^\top+a_d^2u_d^2-{\|\bu_1\|^2\over 1-\|\bu_2\|^2}(\bu_2\bv_2^\top)\right)
    \\&=h_{\mathcal E}^{2k-4}(\bu)\frac{u_d^2\bu_2\bv_2^\top+a_d^2u_d^2(u_d^2+\|\bu_1\|^2)}{1-\|\bu_2\|^2}.
\end{align*}
Finally, combining this with~\eqref{eq:detBlock} and~\eqref{eq:detC4} gives
\begin{align}\label{0821}
    q^{(d)}_{\mathbf{id}}(\bu)&=\prod_{j=1}^{k-1}a^2_jh_{\mathcal E}^{2k-4}(\bu)\left(u_d^2\bu_2\bv_2^\top+a_d^2u_d^2(u_d^2+\|\bu_1\|^2)\right)
    \\\notag
    &=u_d^2h_{\mathcal E}^{2k-4}(\bu)\prod_{j=1}^{k-1}a^2_j\left(\sum_{j=k}^{d}a_j^2u_j^2+a_d^2\sum_{j=1}^{k-1}u_j^2\right)
    \\\notag
    &=u_d^2h_{\mathcal E}^{2k-4}(\bu)\prod_{j=1}^{k-1}a^2_j\left(h_{\mathcal E}(\bu)-\sum_{j=1}^{k-1}a_j^2u_j^2+a_d^2\sum_{j=1}^{k-1}u_j^2\right).
\end{align}
Now consider $q^{(d)}_\tau$ for an arbitrary permutation $\tau\in S_{d-1}$. By the simultaneous rearrangement of the rows and columns  (thus non-changing the determinant) it is possible to bring the matrix corresponding to $q^{(d)}_\tau$ to the same block form as in~\eqref{1945}, and repeating the above reasoning we get
\begin{align*}
    q^{(d)}_\tau(\bu)=u_d^2h_{\mathcal E}^{2k-4}(\bu)\prod_{j=1}^{k-1}a^2_{\tau(j)}\left(h_{\mathcal E}(\bu)-\sum_{j=1}^{k-1}a^2_{\tau(j)}u_{\tau(j)}^2+a_d^2\sum_{j=1}^{k-1}u_{\tau(j)}^2\right).
\end{align*}
Summing this up over $\tau\in S_{d-1}$ leads to
\begin{align*}
    \sum_{\tau\in S_{d-1}}q^{(d)}_\tau(\bu)&=(k-1)!(d-k)!u_d^2h_{\mathcal E}^{2k-4}(\bu) \sum_{\substack{I\subset J_d \\ |I| =k-1}}
    \left[\prod_{j\in I}a^2_j\left(a_d^2u_d^2+\sum_{j\in J_d\setminus I}a^2_ju_j^2+a_d^2\sum_{j\in I}u_j^2\right)\right],
\end{align*}
where we used the notation $J_i:=\{1,\dots,d\}\setminus \{i\}$ for $i=1,\dots,d$. 
Similarly, for arbitrary $i$ we have
\begin{align}\label{1023}
    \sum_{\tau\in S_{d-1}} \!\!\! q_\tau^{(i)}(\bu)&=(k-1)!(d-k)!h_{\mathcal E}^{2k-4}(\bu)
    u_i^2 \!\!\!\! \sum_{\substack{I\subset J_i \\ |I| =k-1}} \!\!\!
    \left[\prod_{j\in I}a^2_j\left(a_i^2u_i^2+\sum_{j\in J_i\setminus I}a^2_ju_j^2+a_i^2\sum_{j\in I}u_j^2\right) \!\! \right] \!\! .
\end{align}
Let us sum this up over $i$ dealing with each summand in the inner brackets separately. To simplify the notation, we will write
\begin{align*}
    \ba^i:=(a_1^2,\dots,a_{i-1}^2,a_{i+1}^2,\dots,a_d^2), &\quad 1\leq i\leq d,
    \\
    \ba^{i,j}:=(a_1^2,\dots,a_{i-1}^2,a_{i+1}^2,\dots,a_{j-1}^2,a_{j+1}^2,\dots,a_d^2),&\quad 1\leq i<j\leq d,
\end{align*}
and $\ba^{i,j}=\ba^{j,i}$ for $i>j$.
Firstly,
\begin{align}\label{1022}
    \sum_{i=1}^d u_i^2
    \sum_{\substack{I\subset J_i \\ |I| =k-1}}
    \left[\prod_{j\in I}a^2_j\cdot (a^2_iu_i^2)\right]
    =
    \sum_{i=1}^da^2_iu_i^4\cdot s_{k-1}(\ba^i).
\end{align}
Secondly, by exchanging the summation order we get
\begin{align*}
    \sum_{i=1}^d u_i^2
    \sum_{\substack{I\subset J_i \\ |I| =k-1}}
    \left[\prod_{j\in I}a^2_j\sum_{j\in J_i\setminus I}
    a^2_ju_j^2\right]&=\sum_{i=1}^d u_i^2
    \sum_{j\in J_i}a^2_ju_j^2\sum_{\substack{I\subset J_{i,j} \\ |I| =k-1}}\prod_{j\in I}a^2_j\\
    &=
    \sum_{\substack{i,j=1 \\ i\ne j}}^da^2_iu_i^2u_j^2\cdot s_{k-1}(\ba^{i,j}),
\end{align*}
where {$J_{i,j}:=\{1,\ldots,d\}\setminus\{i,j\}$ for any $i\neq j$}. 
 Finally, again by exchanging the summation order
\begin{align*}
    \sum_{i=1}^d u_i^2
    \sum_{\substack{I\subset J_i \\ |I| =k-1}}
    \left[\prod_{j\in I}a^2_j\left(a_i^2\sum_{j\in I}u_j^2\right)\right]
        &=
    \sum_{i=1}^d a_i^2u_i^2\sum_{j\in J_i}a_j^2u_j^2
    \sum_{\substack{I\subset J_{i,j} \\ |I| =k-2}}
    \prod_{j\in I}a^2_j
    \\&=
    \sum_{\substack{i,j=1 \\ i\ne j}}^da^2_iu_i^2a_j^2u_j^2\cdot s_{k-2}(\ba^{i,j}).
\end{align*}
Further we note that
\begin{align*}
    s_{k-1}(\ba^i)=s_{k-1}(\ba^{i,j})+a_j^2s_{k-2}(\ba^{i,j})
\end{align*}
and thus
\begin{align*}
    \sum_{i=1}^d u_i^2
    \sum_{\substack{I\subset J_i \\ |I| =k-1}}
    \left[\prod_{j\in I}a^2_j\sum_{j\in J_i\setminus I}
    a^2_ju_j^2\right]
    +
    \sum_{i=1}^d u_i^2
    \sum_{\substack{I\subset J_i \\ |I| =k-1}}
    \left[\prod_{j\in I}a^2_j\left(a_i^2\sum_{j\in I}u_j^2\right)\right]
    \\
        =\sum_{\substack{i,j=1 \\ i\ne j}}^da^2_iu_i^2u_j^2\cdot (s_{k-1}(\ba^{i,j})+a_j^2s_{k-2}(\ba^{i,j}))
        =
        \sum_{\substack{i,j=1 \\ i\ne j}}^da^2_iu_i^2u_j^2\cdot s_{k-1}(\ba^i).
\end{align*}
Combining this with~\eqref{1023} and~\eqref{1022} we arrive at
\begin{align*}
    \sum_{i=1}^d\sum_{\tau\in S_{d-1}}q_\tau^{(i)}(\bu)
    &=
    (k-1)!(d-k)!h_{\mathcal E}^{2k-4}(\bu)
    \left(\sum_{i=1}^da^2_iu_i^4\cdot s_{k-1}(\ba^i)+\sum_{\substack{i,j=1 \\ i\ne j}}^da^2_iu_i^2u_j^2\cdot s_{k-1}(\ba^i)\right)
    \\&=
    (k-1)!(d-k)!h_{\mathcal E}^{2k-4}(\bu)\sum_{i=1}^da^2_iu_i^2s_{k-1}(\ba^i),
\end{align*}
where in the last step we used that $\sum_{i=1}^d u_i^2=1$. Recalling~\eqref{2038}--\eqref{1605} concludes the proof.

\subsection{Proof of Proposition~\ref{0831}}
The key ingredient of the proof is the following simple observation: for $c,\gamma>0$ we have
\begin{align}\label{1113}
    \int_{0}^{\infty}t^{\gamma-1} e^{-ct^2}\dd t
    =
    \frac{c^{-\gamma/2}}2\int_{0}^{\infty}(ct^2)^{\gamma/2-1} e^{-ct^2}\dd (ct^2)
    =
    \frac{c^{-\gamma/2}}2\Gamma\bigg(\frac{\gamma}2\bigg).
\end{align}
Passing  to the spherical coordinates and having in mind that $h_{\mathcal E}(\cdot)$ is homogeneous of degree 1 we can write 
\begin{align}\label{1137}
     \int\limits_{\R^d}|x_i|^\alpha h_{\mathcal E}^{-\beta}(\bx)e^{-\|\bx\|^2} \dd\bx&=\int\limits_{\s^{d-1}}|u_i|^\alpha h_{\mathcal E}^{-\beta}(\bu)\int_0^\infty r^{d+\alpha-\beta-1} e^{-r^2}\dd r\sigma(\dd\bu)
     \\\notag
     &=\frac12\Gamma\bigg({d+\alpha-\beta\over 2}\bigg)\int\limits_{\s^{d-1}}|u_i|^\alpha h_{\mathcal E}^{-\beta}(\bu)\sigma(\dd\bu),
\end{align}
where in the last step we used~\eqref{1113} with $c=1, \gamma=d+\alpha-\beta$. Now using~\eqref{1137} and applying~\eqref{1113} with $c=h_{\mathcal E}^{2}(\bx), \gamma=\beta$ leads to
\begin{align}\label{1507}
     \int\limits_{\s^{d-1}}|u_i|^\alpha &h^{-\beta}_{\mathcal E}(\bu)\,\sigma(\dd\bu)={2\over \Gamma({d+\alpha-\beta\over 2})}\int\limits_{\R^d}|x_i|^\alpha h_{\mathcal E}^{-\beta}(\bx)e^{-\|\bx\|^2} \dd\bx
    \\\notag
    &={2\over \Gamma({d+\alpha-\beta\over 2})}
    \int\limits_{\R^d}
    \bigg[{2\over \Gamma({\beta\over 2})}\int_{0}^{\infty}t^{\beta-1}e^{-h_{\mathcal E}^2(\bx)t^2}\dd t\bigg]|x_i|^\alpha e^{-\|\bx\|^2} \dd\bx
    \\\notag
    &={4\over \Gamma({d+\alpha-\beta\over 2})\Gamma({\beta\over 2})}\int_{0}^{\infty}t^{\beta-1}\bigg[\int_{-\infty}^{\infty}|x_i|^\alpha e^{-(a_i^2t^2+1)x_i^2}\dd x_i\prod_{j\neq i}\int_{-\infty}^{\infty}e^{-(a_j^2t^2+1)x_j^2}\dd x_j\bigg]\dd t,
\end{align}
where in the last step we used that $h_{\mathcal E}^2(\bx)=a^2_1x^2_1+\dots+a^2_dx^2_d$ and the Fubini theorem. Applying~\eqref{1113}  to the inner integrals gives
\begin{align*}
    \int_{-\infty}^{\infty}|x_i|^\alpha e^{-(a_i^2t^2+1)x_i^2}\dd x_i
    =
    2\int_{0}^{\infty}x_i^\alpha e^{-(a_i^2t^2+1)x_i^2}\dd x_i
    =
    {\Gamma(\frac{\alpha+1}{2})\over (a_i^2t^2+1)^{(\alpha+1)/2}}
\end{align*}
and
\begin{align*}
    \int_{-\infty}^{\infty}e^{-(a_j^2t^2+1)x_j^2}\dd x_j
    =
    2\int_{0}^{\infty} e^{-(a_j^2t^2+1)x_j^2}\dd x_j
    =
    {\sqrt\pi \over (a_j^2t^2+1)^{1/2}},
\end{align*}
which together with~\eqref{1507} concludes the proof.

\subsection{Proof of Theorem~\ref{0948}}
Theorem~\ref{0948} follows from Theorem~\ref{1406} and Proposition~\ref{0831} applied with $\alpha=2,\beta=k$, and the observation that for such $\alpha, \beta$ we have
\begin{align*}
    {1\over k\kappa_{d-k}}\cdot
     {4\pi^{(d-1)/2} \Gamma(\frac{3}{2})\over \Gamma({d+2-k\over 2})\Gamma({k\over 2})}=\kappa_k.
\end{align*}

\section*{Acknowledgements}
The authors thank Jairo Bochi from Pennsylvania State University for his reference suggestions and helpful comments on hypergeometric functions.

\bibliographystyle{plain}

\end{document}